\documentclass[11pt]{article}

\usepackage[T1]{fontenc}
\usepackage[utf8]{inputenc}
\usepackage{lmodern}
\usepackage{amsmath,amssymb,amsthm,mathtools}
\usepackage{booktabs,longtable,array}
\usepackage[margin=0.9in]{geometry}
\usepackage{microtype}
\usepackage[hidelinks]{hyperref}

\usepackage{multicol}
\usepackage[bottom]{footmisc}
\hypersetup{
  pdftitle={Translational tiles without spectra in finite Abelian p-groups},
  pdfauthor={Shilei Fan}
}

\newtheorem{theorem}{Theorem}[section]
\newtheorem{proposition}[theorem]{Proposition}
\newtheorem{lemma}[theorem]{Lemma}

\newtheorem{remark}[theorem]{Remark}
\newtheorem{conjecture}[theorem]{Conjecture}
\theoremstyle{definition}
\newtheorem{definition}[theorem]{Definition}

\newcommand{\F}{\mathbb F}
\newcommand{\Z}{\mathbb Z}
\newcommand{\C}{\mathbb C}
\newcommand{\1}{\mathbf 1}
\newcommand{\ii}{\mathrm i}
\newcommand{\pp}[2]{\ensuremath{(#1;#2)}}

\providecommand{\keywords}[1]{\textbf{Keywords:} #1}
\providecommand{\subjclass}[1]{\textbf{MSC (2020):} #1}

\title{Translational tiles without spectra\\
in finite abelian $p$-groups}
\author{Shilei Fan\\
\small School of Mathematics and Statistics, Central China Normal University\\
\small Wuhan 430079, P.R. China\\
\small \texttt{slfan@mail.ccnu.edu.cn}\\
\and
Mamateli Kadir\\
\small School of Mathematics and Statistics, Kashi University\\
\small Kashi 844000, P.R. China\\
\small \texttt{mamatili880@163.com}}
\date{}

\begin{document}
\maketitle

\begin{abstract}
We construct explicit translational tiles without spectra in three finite
abelian $p$-groups.  The first is a $64$-point subset of
$\Z_4^4\times\Z_2^2$.  The other two are a $512$-point subset of
$\F_2^{13}$ and a $2187$-point subset of $\F_3^9$.  Consequently, the
tile-to-spectral implication fails for finite abelian $p$-groups, and it
already fails within the class of elementary abelian groups for both
$p=2$ and $p=3$.

Two elementary mechanisms organize the examples.  A two-layer obstruction
turns a spectral non-tile with two suitable tiling complements into a tile
without a spectrum.  A fiber--clique obstruction converts a family of
tiling complements with controlled common Fourier zeros into an elementary
abelian counterexample.  All coordinate data are included.  The finite
claims are certified by three short, self-contained programs using exact
integer arithmetic and exhaustive searches; the accompanying source files
recompute every assertion used in the proofs.
\end{abstract}

\keywords{Translational tiles, spectral sets, abelian $p$-groups, Fuglede's conjecture.}

\subjclass{20K01, 05B45, 42A85, 11B75.}

\section{Introduction}

Fuglede's conjecture connects translational tilings with orthogonal bases of
exponentials~\cite{Fuglede}.  Its analogue for a finite abelian group $G$
asks whether a subset of $G$ tiles by translations precisely when it admits
an orthogonal basis of characters.  The two implications are logically
different, and we denote them by T--S (tile implies spectral) and S--T
(spectral implies tile).

Fuglede formulated the so called spectral set conjecture on Euclidean space in 1974 and proved it when
either the spectrum or the tiling complement is a lattice~\cite{Fuglede}.
For almost three decades the known results were compatible with the
conjecture.  The finite-group viewpoint then became decisive: Tao found a
six-point spectral non-tile in $\F_3^5$ and transferred it to a counterexample
in $\mathbb R^5$~\cite{Tao}.  Soon afterward, finite constructions produced
counterexamples in both directions that lift to dimension three: a spectral
non-tile in $\Z_8^3$~\cite{KM-hadamard} and a tile without a spectrum in
$\Z_{24}^3$~\cite{FMM}.  Thus finite Fourier analysis became a central source
of examples for the Euclidean as well as the discrete problem.

After the general equivalence failed, attention shifted to restricted group
classes and to the two implications separately.  The classical
$\Z_{24}^3$ counterexample to T--S lives in a group that is not a $p$-group,
whereas
the $p$-group example in $\Z_8^3$ concerns S--T.  The latter direction also
fails in several elementary abelian groups; for every odd prime $p$ it fails
in $\F_p^4$~\cite{FergusonSothanaphan}.  In contrast, T--S was conjectured to
hold in every finite abelian $p$-group.  Malikiosis formulated the following
statement explicitly~\cite[Conjecture~1.3]{Malikiosis-LP}.

\begin{conjecture}\label{conj:pgroup}
Every translational tile in a finite abelian $p$-group is spectral.
\end{conjecture}

The conjecture was supported by many positive results.  It holds in cyclic
$p$-groups~\cite{Laba,Malikiosis-cyclic} and in
$\Z_p\times\Z_{p^n}$~\cite{Shi-ranktwo,Zhang-ranktwo}.  The full finite
Fuglede conjecture holds in $\F_p^2$~\cite{IMP}, while T--S holds in
$\F_p^3$ for every prime $p$; see~\cite{Malikiosis-LP,FMMV} and the
references therein.  Fan and Zhang developed a periodic-tiling method that
yields additional positive families~\cite{FanZhang}.  Their public preprint
also announces further elementary cases as forthcoming; we do not use those
announcements as proved results here.

Our main result disproves Conjecture~\ref{conj:pgroup}, first in a small
exponent-$4$ group and then in elementary abelian groups.

\begin{theorem}\label{thm:main}
There exist explicit translational tiles without spectra in each of the
following groups:
\begin{enumerate}
\item a $64$-point set $\Gamma\subset\Z_4^4\times\Z_2^2$;
\item a $512$-point set $R_2\subset\F_2^{13}$;
\item a $2187$-point set $R_3\subset\F_3^9$.
\end{enumerate}
The corresponding tiling complements have cardinalities $16$, $16$, and
$9$, respectively.
\end{theorem}

The last two examples show that the failure is not caused merely by the
presence of elements of order $p^2$.  They give the dimension upper bounds
$d_2\leq13$ and $d_3\leq9$ for the first possible elementary abelian
counterexamples.  We do not claim that either dimension is minimal.

The first example is driven by a $16$-point spectral non-tile in
$\Z_4^4\times\Z_2$ and a short two-layer lemma.  The elementary examples
are instances of a second common construction: complements of one small
tile are placed in layers, and spectrality of the resulting large tile
would force a forbidden clique in an exact Fourier-zero Cayley graph.
These mechanisms are related to the layered finite-group constructions
in~\cite{KM-tiles,KM-hadamard}, but the two obstruction lemmas below isolate
the precise certificates needed here.

The paper is organized as follows.  Section~\ref{sec:prelim} recalls the
finite Fourier criteria.  Sections~\ref{sec:twolayer}
and~\ref{sec:fiber} isolate the two construction principles.  The three
examples are proved in Sections~\ref{sec:exp4},~\ref{sec:f2},
and~\ref{sec:f3}.  Section~\ref{sec:verification} explains the exact
certificates.  Complete coordinate data appear in the appendices.

\section{Finite Fourier preliminaries}\label{sec:prelim}

All groups below are finite, abelian, and written additively.  Let $G$ be
such a group and let $\widehat G$ denote its character group.

\begin{definition}
A set $A\subset G$ is a \emph{translational tile} if there is a set
$B\subset G$ such that every element of $G$ has a unique representation
$a+b$ with $a\in A$ and $b\in B$.  We then write $A\oplus B=G$ and call
$B$ a tiling complement of $A$.
\end{definition}

If $|A||B|=|G|$, the direct-sum condition is equivalent to
\begin{equation}\label{eq:difference-test}
 A\oplus B=G
 \quad\Longleftrightarrow\quad
 (A-A)\cap(B-B)=\{0\}.
\end{equation}

For $f:G\to\C$ we use the Fourier transform
\[
 \widehat f(\chi)=\sum_{x\in G}f(x)\overline{\chi(x)},
 \qquad \chi\in\widehat G.
\]
For $A\subset G$, write
\[
 Z(A)=\{\chi\in\widehat G\setminus\{0\}:
             \widehat{\1_A}(\chi)=0\}.
\]

\begin{definition}
A set $A\subset G$ is \emph{spectral} if there is a set
$\Lambda\subset\widehat G$ with $|\Lambda|=|A|$ such that the characters
in $\Lambda$, restricted to $A$, form an orthogonal basis of $\C^A$.
The set $\Lambda$ is then called a spectrum of $A$.
\end{definition}

Equivalently, $\Lambda$ is a spectrum of $A$ precisely when
$|\Lambda|=|A|$ and
\begin{equation}\label{eq:spectral-test}
 (\Lambda-\Lambda)\setminus\{0\}
 \subseteq
 \{\chi\in\widehat G:\widehat{\1_A}(\chi)=0\}.
\end{equation}
Here and below we use additive notation for character groups after fixing a
self-dual identification.  For $\F_p^n$ the identification is given by
$\chi_v(x)=\exp(2\pi\ii v\cdot x/p)$.

\section{A two-layer obstruction}\label{sec:twolayer}

The following lemma contains the conceptual part of the exponent-$4$
example.

\begin{lemma}\label{lem:two-layer}
Let $H$ be a finite abelian group.  Suppose $K\subset H$ has a spectrum
$E\subset\widehat H$ but does not tile $H$.  Suppose also that
\[
 E\oplus T_0=\widehat H,
 \qquad E\oplus T_1=\widehat H,
\]
and that for every $d\in(K-K)\setminus\{0\}$,
\begin{equation}\label{eq:magnitude-separation}
 \bigl|\widehat{\1_{T_0}}(d)\bigr|
 \ne
 \bigl|\widehat{\1_{T_1}}(d)\bigr|.
\end{equation}
Then
\[
 \Gamma=(T_0\times\{0\})\cup(T_1\times\{1\})
 \subset\widehat H\times\Z_2
\]
tiles $\widehat H\times\Z_2$ and is not spectral.
\end{lemma}

\begin{proof}
The two direct-sum identities give
\[
 \Gamma\oplus(E\times\{0\})=\widehat H\times\Z_2,
\]
so $\Gamma$ tiles.  For $d\in H$,
\[
 \widehat{\1_\Gamma}(d,0)
 =\widehat{\1_{T_0}}(d)+\widehat{\1_{T_1}}(d).
\]
Condition~\eqref{eq:magnitude-separation} implies that this coefficient is
nonzero whenever $0\ne d\in K-K$.

Suppose $\Lambda\subset H\times\Z_2$ were a spectrum of $\Gamma$.
The spectral criterion gives
\[
 ((\Lambda-\Lambda)\setminus\{0\})
 \cap ((K-K)\times\{0\})=\varnothing.
\]
Since $|E|=|K|$ and $|T_j|=|H|/|E|$, we have
$|K||\Lambda|=2|H|$.  The difference test therefore yields
\[
 \Lambda\oplus(K\times\{0\})=H\times\Z_2.
\]
Restricting this tiling to either $\Z_2$-layer would make $K$ tile $H$,
contrary to the hypothesis.
\end{proof}

\section{The fiber--clique obstruction}\label{sec:fiber}

The next lemma packages both elementary abelian constructions.

\begin{lemma}\label{lem:fiber-clique}
Let $H$ and $L$ be finite abelian groups.  Let $A\subset H$, and for each
$y\in L$ let $C_y\subset H$ satisfy
\[
 A\oplus C_y=H.
\]
Put
\[
 R=\bigcup_{y\in L}(C_y\times\{y\})\subset H\times L,
 \qquad q=|C_y|=\frac{|H|}{|A|}.
\]
For a nontrivial $\chi\in\widehat H$, define
\[
 W(\chi)=\sum_{y\in L}\widehat{\1_{C_y}}(\chi),
 \qquad
 D=\{\chi\ne0:W(\chi)=0\}.
\]
Let $\mathcal G_D$ be the Cayley graph on $\widehat H$ in which distinct
$u,v$ are adjacent when $u-v\in D$.  If $\mathcal G_D$ has no clique of
size $q$, then $R$ tiles $H\times L$ but is not spectral.
\end{lemma}

\begin{proof}
Layer by layer,
\[
 (A\times\{0\})\oplus R=H\times L,
\]
so $R$ tiles.  Moreover, at a horizontal frequency $(\chi,0)$,
\begin{equation}\label{eq:horizontal-fourier}
 \widehat{\1_R}(\chi,0)
 =\sum_{y\in L}\widehat{\1_{C_y}}(\chi)=W(\chi).
\end{equation}

If $\Lambda\subset\widehat H\times\widehat L$ were a spectrum of $R$,
then
\[
 |\Lambda|=|R|=q|L|.
\]
Partition $\Lambda$ according to its second coordinate.  One of the $|L|$
fibers contains at least $q$ elements.  The first coordinates of any $q$
such elements are distinct, and their pairwise differences lie in $D$ by
\eqref{eq:spectral-test} and~\eqref{eq:horizontal-fourier}.  They form a
$q$-clique in $\mathcal G_D$, a contradiction.
\end{proof}

\begin{remark}\label{rem:types}
In the examples only finitely many distinct sets occur among the $C_y$.
If $C_j$ occurs on a block $Y_j\subset L$ of cardinality $a_j$, then
\[
 W(\chi)=\sum_j a_j\widehat{\1_{C_j}}(\chi).
\]
Thus the certificate consists of common tilings, an exact weighted Fourier
zero set, and a clique obstruction in one Cayley graph.
\end{remark}

\section{A \texorpdfstring{$64$}{64}-point example in an
\texorpdfstring{exponent-$4$}{exponent-4} group}\label{sec:exp4}

Let
\[
 H=\Z_4^4\times\Z_2
\]
and identify $H$ with its dual through
\begin{equation}\label{eq:pairing-z4}
 \langle x,y\rangle
 =x_1y_1+x_2y_2+x_3y_3+x_4y_4+2x_5y_5\pmod4,
 \qquad \chi_y(x)=\ii^{\langle x,y\rangle}.
\end{equation}
The notation $\pp{abcd}{e}$ abbreviates $(a,b,c,d;e)$.

Appendix~\ref{app:exp4-data} gives sets
\[
 K,E,T_0,T_1\subset H,
 \qquad |K|=|E|=16,\qquad |T_0|=|T_1|=32.
\]

\begin{proposition}\label{prop:exp4-cert}
The sets in Appendix~\ref{app:exp4-data} satisfy the following properties.
\begin{enumerate}
\item $E$ is a spectrum of $K$;
\item $E\oplus T_0=H$ and $E\oplus T_1=H$;
\item $K$ does not tile $H$;
\item for every $0\ne d\in K-K$,
\[
 |\widehat{\1_{T_0}}(d)|^2\ne
 |\widehat{\1_{T_1}}(d)|^2.
\]
\end{enumerate}
\end{proposition}

\begin{proof}
All assertions are exact finite statements.  For (1), the verifier forms
the $16\times16$ matrix
${(\ii^{\langle k,e\rangle})}_{k\in K,e\in E}$ and checks that its Gram
matrix is $16I_{16}$.  For (2), it checks all sums and their uniqueness in
$H$.  For (3), it normalizes a hypothetical complement to contain $0$ and
exhausts the resulting exact-cover tree.  For (4), all Fourier coefficients
are Gaussian integers; their squared norms are compared exactly for every
nonzero element of $K-K$.  The complete implementation is the ancillary
file \path{verify_finite_2group_tile_nonspectral.py}.
Section~\ref{sec:verification} records its output and explains the audit.
\end{proof}

\begin{proof}[Proof of Theorem~\ref{thm:main}, part~(1)]
Apply Lemma~\ref{lem:two-layer} to Proposition~\ref{prop:exp4-cert}.  The
resulting set
\[
 \Gamma=(T_0\times\{0\})\cup(T_1\times\{1\})
 \subset H\times\Z_2\cong\Z_4^4\times\Z_2^2
\]
has $64$ points, tiles with complement $E\times\{0\}$, and is not
spectral.
\end{proof}

\section{An elementary binary example}\label{sec:f2}

Identify $\F_2^d$ with the integers $0,\ldots,2^d-1$ by writing
$x=\sum_{i=0}^{d-1}x_i2^i$ for $(x_0,\ldots,x_{d-1})$.  Addition of codes
means coordinatewise addition modulo two.

Let $H_2=\F_2^8$.  Appendix~\ref{app:f2-data} gives a $16$-point set
$A_2\subset H_2$, seventeen $16$-point sets $C_0,\ldots,C_{16}$, and
positive weights
\begin{equation}\label{eq:f2-weights}
 (a_0,\ldots,a_{16})=(1,1,2,1,1,1,3,1,1,2,3,3,1,5,3,2,1),
 \qquad \sum_j a_j=32.
\end{equation}
Partition $L_2=\F_2^5$, in increasing integer-code order, into consecutive
blocks $Y_j$ with $|Y_j|=a_j$, and define
\begin{equation}\label{eq:R2}
 R_2=\bigcup_{j=0}^{16}(C_j\times Y_j)
 \subset\F_2^8\times\F_2^5=\F_2^{13}.
\end{equation}

For $v\in H_2$, put
\[
 W_2(v)=\sum_{j=0}^{16}a_j\widehat{\1_{C_j}}(v),
 \qquad
 D_2=\bigcap_{j=0}^{16}Z(C_j).
\]

\begin{proposition}\label{prop:f2-cert}
The data in Appendix~\ref{app:f2-data} satisfy:
\begin{enumerate}
\item $A_2\oplus C_j=H_2$ for $0\le j\le16$;
\item $|D_2|=108$, and for every nonzero $v\in H_2$,
\[
 W_2(v)=0\quad\Longleftrightarrow\quad v\in D_2;
\]
\item the Cayley graph $\mathcal G_{D_2}$ has clique number $12$;
\item the minimum of $|W_2(v)|$ over
$v\notin D_2\cup\{0\}$ is $4$.
\end{enumerate}
\end{proposition}

\begin{proof}
The ancillary C++ verifier checks every sum in (1), evaluates all Walsh
Fourier coefficients in integer arithmetic, and obtains (2) and (4).
For (3) it uses an exact coloring branch-and-bound algorithm on $256$
vertices.  It exhibits a $12$-clique and exhaustively rules out a
$13$-clique in $425$ search nodes.  An independent binary MILP formulation
returns the same optimum.
\end{proof}

\begin{proof}[Proof of Theorem~\ref{thm:main}, part~(2)]
Apply Lemma~\ref{lem:fiber-clique} with $H=H_2$, $L=L_2$, $A=A_2$, and
$C_y=C_j$ for $y\in Y_j$.  Here $q=16$, whereas
$\omega(\mathcal G_{D_2})=12$.  Hence $R_2$ tiles $\F_2^{13}$ with
complement $A_2\times\{0\}$ and is not spectral.  Its cardinality is
$16\cdot32=512$.
\end{proof}

\section{An elementary ternary example}\label{sec:f3}

Identify $\F_3^d$ with the integers $0,\ldots,3^d-1$ by writing
$x=\sum_{i=0}^{d-1}x_i3^i$ for $(x_0,\ldots,x_{d-1})$.  Addition of codes
means coordinatewise addition modulo three.

Let $H_3=\F_3^6$.  Appendix~\ref{app:f3-data} gives a $9$-point set
$A_3\subset H_3$, fifteen $81$-point sets $C_0,\ldots,C_{14}$, and
positive weights
\begin{equation}\label{eq:f3-weights}
 (a_0,\ldots,a_{14})=(2,2,4,4,1,1,1,2,1,1,1,3,2,1,1),
 \qquad \sum_j a_j=27.
\end{equation}
Partition $L_3=\F_3^3$, in increasing integer-code order, into consecutive
blocks $Y_j$ with $|Y_j|=a_j$, and define
\begin{equation}\label{eq:R3}
 R_3=\bigcup_{j=0}^{14}(C_j\times Y_j)
 \subset\F_3^6\times\F_3^3=\F_3^9.
\end{equation}

Let $\zeta=e^{2\pi\ii/3}$ and, for $v\in H_3$, put
\[
 W_3(v)=\sum_{j=0}^{14}a_j\widehat{\1_{C_j}}(v),
 \qquad
 D_3=\bigcap_{j=0}^{14}Z(C_j).
\]

\begin{proposition}\label{prop:f3-cert}
The data in Appendix~\ref{app:f3-data} satisfy:
\begin{enumerate}
\item $A_3\oplus C_j=H_3$ for $0\le j\le14$;
\item $|D_3|=582$, and for every nonzero $v\in H_3$,
\[
 W_3(v)=0\quad\Longleftrightarrow\quad v\in D_3;
\]
\item the Cayley graph $\mathcal G_{D_3}$ has no clique of size $81$;
\item the minimum nonzero Eisenstein norm $|W_3(v)|^2$ over
$v\notin D_3\cup\{0\}$ is $2187$.
\end{enumerate}
\end{proposition}

\begin{proof}
For a Fourier coefficient let
$n_r=|\{x\in C_j:v\cdot x=r\}|$.  The verifier represents
$n_0+n_1\zeta+n_2\zeta^2$ by the integer pair
$(n_0-n_2,n_1-n_2)$, so no floating-point arithmetic is used.  It checks
all $15\cdot729$ tiling sums, evaluates (2) and (4), and applies an exact
coloring branch-and-bound search to (3).  The no-$81$-clique search has
$76{,}620$ nodes.
\end{proof}

\begin{proof}[Proof of Theorem~\ref{thm:main}, part~(3)]
Apply Lemma~\ref{lem:fiber-clique} with $H=H_3$, $L=L_3$, $A=A_3$, and
$C_y=C_j$ for $y\in Y_j$.  Here $q=81$, and
$\mathcal G_{D_3}$ has no $81$-clique.  Hence $R_3$ tiles $\F_3^9$ with
complement $A_3\times\{0\}$ and is not spectral.  Its cardinality is
$81\cdot27=2187$.
\end{proof}

\section{Exact verification and reproducibility}\label{sec:verification}

The theoretical reductions in Lemmas~\ref{lem:two-layer}
and~\ref{lem:fiber-clique} are independent of computation.  The three
propositions are finite certificates attached to explicit integer data.
They are verified by the following ancillary source files:
\begin{center}
\begin{tabular}{@{}ll@{}}
\toprule
example & verifier\\
\midrule
$\Z_4^4\times\Z_2^2$ &
\path{verify_finite_2group_tile_nonspectral.py}\\
$\F_2^{13}$ &
\path{verify_elementary_f2_tile_nonspectral.cpp}\\
$\F_3^9$ &
\path{verify_elementary_f3_tile_nonspectral.cpp}\\
\bottomrule
\end{tabular}
\end{center}
The programs use only exact integer arithmetic.  The clique searches are
finite branch-and-bound enumerations with valid greedy-coloring upper
bounds.  The non-tiling search in the first example is an exhaustive exact
cover search after translation normalization.

On the coordinate data printed in the appendices, the three programs return
the following decisive output:
\begin{verbatim}
VERIFIED H=C4^4xC2 order=512
K_spectral=1 K_tile=0 exact_cover_nodes=42
E_plus_T0=H E_plus_T1=H
Fourier_magnitude_separated=1 oriented_differences=152
Gamma_tile=1 Gamma_spectral=0 ambient=C4^4xC2^2 size=64

VERIFIED_ELEMENTARY_TILE_NONSPECTRAL
group=F2^13 group_order=8192 tile_size=512 tiling_complement_size=16
common_fourier_zeros=108
common_zero_clique_number=12 no_16_clique=1
weighted_noncancellation=1 minimum_abs_weighted_fourier=4

VERIFIED_F3_9_TILE_NONSPECTRAL tile_size=2187 complement_size=9
common_zero=582 no_81_clique_nodes=76620
minimum_nonzero_eisenstein_norm=2187
\end{verbatim}

For reproduction, run
\begin{verbatim}
python3 verify_finite_2group_tile_nonspectral.py
clang++ -O3 -std=c++20 verify_elementary_f2_tile_nonspectral.cpp \
  -o verify_f2 && ./verify_f2
clang++ -O3 -std=c++20 verify_elementary_f3_tile_nonspectral.cpp \
  -o verify_f3 && ./verify_f3
\end{verbatim}
The first program is deliberately short and works directly in the five
displayed coordinates.  The other two recompute their Fourier zero sets
from the printed complement lists before starting the graph searches.

\medskip
\noindent\textbf{AI-use disclosure.}
The search for the examples in this paper was carried out with the
assistance of OpenAI's GPT-5.6.  The model helped formulate and implement
the exploratory searches that produced the candidates.  All mathematical
claims were subsequently checked by the authors and are independently
certified by the exact programs described above.

\textbf{Fundings:} S. L. FAN  is  partially supported by NSFC (grants No. 12331004  and No.  12231013).
M. Kadir is supported by NSF of China (Grant No. 12361015), and by NSF of Xinjiang Uygur Autonomous Region,
P. R. China (Grant No. 2025D01A09).

\section{Concluding remarks}

The examples answer the $p$-group T--S conjecture negatively and show that
the obstruction persists in elementary abelian groups.  The constructions
do not determine the least possible ambient orders or dimensions.  In
particular, the numbers $13$ and $9$ arise from the present base-plus-layer
certificates rather than from lower-bound theorems.

The fiber--clique lemma suggests two routes to smaller examples.  One may
reduce the dimension of the base group while retaining enough tiling
complements, or reduce the number of layer types while forcing the weighted
horizontal Fourier zero graph to have clique number below the complement
size.  Both are finite but structurally constrained search problems.

\appendix

\section{Coordinate data for the
\texorpdfstring{exponent-$4$}{exponent-4} example}\label{app:exp4-data}

In $H=\Z_4^4\times\Z_2$, define
\begin{align*}
K={}&\{\pp{0000}{0},\pp{0200}{0},\pp{0000}{1},\pp{1000}{0},
\pp{3202}{0},\pp{0100}{0},\pp{0010}{0},\pp{2112}{1},\\
&\hspace{1.4em}\pp{2122}{0},\pp{0220}{1},\pp{0001}{0},\pp{3001}{0},
\pp{1003}{0},\pp{2203}{0},\pp{0030}{0},\pp{2132}{1}\},\\
E={}&\{\pp{0000}{0},\pp{0012}{0},\pp{0221}{0},\pp{2231}{0},
\pp{0203}{1},\pp{2213}{1},\pp{2020}{1},\pp{2032}{1},\\
&\hspace{1.4em}\pp{1133}{0},\pp{2303}{0},\pp{2311}{0},\pp{3123}{0},
\pp{0323}{1},\pp{0331}{1},\pp{1111}{1},\pp{3101}{1}\},\\
T_0=\{&\pp{0000}{0},\pp{2002}{0},\pp{2301}{1},\pp{0303}{1},
\pp{1301}{0},\pp{3303}{0},\pp{1320}{0},\pp{3322}{0},\\
&\pp{0300}{0},\pp{2302}{0},\pp{0021}{0},\pp{2023}{0},
\pp{1023}{1},\pp{3021}{1},\pp{0023}{0},\pp{2021}{0},\\
&\pp{3323}{1},\pp{1321}{1},\pp{2321}{0},\pp{0323}{0},
\pp{0002}{0},\pp{2000}{0},\pp{0302}{0},\pp{2300}{0},\\
&\pp{1003}{0},\pp{3001}{0},\pp{1022}{0},\pp{3020}{0},
\pp{3302}{1},\pp{1300}{1},\pp{3022}{0},\pp{1020}{0}\},\\[1mm]
T_1=\{&\pp{0000}{0},\pp{3100}{0},\pp{0301}{0},\pp{1223}{1},
\pp{1322}{1},\pp{0210}{0},\pp{0230}{0},\pp{1013}{1},\\
&\pp{1033}{1},\pp{0200}{1},\pp{2123}{1},\pp{1023}{0},
\pp{3121}{1},\pp{1123}{1},\pp{1301}{0},\pp{3303}{0},\\
&\pp{2320}{1},\pp{0322}{1},\pp{0100}{0},\pp{2102}{0},
\pp{1213}{1},\pp{1223}{0},\pp{1233}{1},\pp{3003}{0},\\
&\pp{1203}{0},\pp{1001}{0},\pp{0222}{1},\pp{2220}{1},
\pp{0000}{1},\pp{2002}{1},\pp{0002}{0},\pp{2000}{0}\}.
\end{align*}

\section{Coordinate data for the binary example}\label{app:f2-data}

The base tile is
\[
 A_2=\{3,22,35,53,69,74,96,101,136,144,168,179,199,204,230,231\}
 \subset\F_2^8.
\]
The complements are:
\begingroup
\scriptsize
\setlength\LTleft{0pt}
\setlength\LTright{0pt}
\begin{longtable}{@{}c>{\raggedright\arraybackslash}p{0.88\textwidth}@{}}
\toprule
$j$ & $C_j$ (binary integer codes)\\
\midrule
\endhead
0&0,9,36,45,48,57,72,97,150,159,199,215,222,238,243,250\\
1&0,9,36,45,48,57,72,81,88,97,117,124,150,159,199,238\\
2&0,9,36,45,48,51,57,58,72,81,88,97,199,208,217,238\\
3&0,9,36,45,48,51,57,58,72,97,114,123,199,238,243,250\\
4&0,9,36,45,48,57,72,97,181,188,199,208,215,217,222,238\\
5&0,9,19,36,45,58,72,97,156,181,199,208,217,238,244,253\\
6&0,9,19,26,36,45,114,123,133,140,178,181,187,188,199,206\\
7&0,9,36,39,45,46,117,124,149,150,156,159,178,187,199,206\\
8&0,26,40,50,82,96,104,122,142,148,166,188,198,206,220,244\\
9&0,40,49,57,96,104,113,121,142,151,159,166,198,206,215,223\\
10&0,26,39,61,71,90,96,125,129,155,166,188,198,219,225,252\\
11&0,25,30,39,71,94,96,121,129,166,184,191,198,223,225,248\\
12&0,30,44,50,96,108,114,126,138,148,166,184,198,202,212,216\\
13&0,8,23,31,77,82,90,101,166,174,177,185,195,235,244,252\\
14&0,2,28,30,71,89,101,123,164,166,184,186,195,221,225,255\\
15&0,20,39,51,94,106,109,121,129,149,166,178,203,204,223,248\\
16&0,6,18,20,88,94,106,108,160,166,178,180,202,204,248,254\\
\bottomrule
\end{longtable}
\endgroup
The layer blocks are determined by the weights
in~\eqref{eq:f2-weights}: $Y_j$ is the consecutive integer-code interval
beginning at
$\sum_{i<j}a_i$ and ending at $\sum_{i\le j}a_i-1$.

\section{Coordinate data for the ternary example}\label{app:f3-data}

The base tile is
\[
 A_3=\{0,269,401,420,465,635,649,658,670\}\subset\F_3^6.
\]
The complements are:
\begingroup
\scriptsize
\catcode`\,=\active
\def,{\char44\allowbreak}
\setlength\LTleft{0pt}
\setlength\LTright{0pt}
\begin{longtable}{@{}c>{\raggedright\arraybackslash}p{0.93\textwidth}@{}}
\toprule
$j$&$C_j$ (ternary integer codes)\\
\midrule
\endhead
0&0,14,25,27,41,52,54,68,79,86,97,99,113,124,126,140,151,153,169,171,185,196,198,212,223,225,239,250,252,266,277,279,293,304,306,320,324,338,349,351,365,376,378,392,403,410,421,423,437,448,450,464,475,477,491,502,504,518,529,531,545,556,558,574,576,590,601,603,617,628,630,644,648,662,673,675,689,700,702,716,727\\
1&0,3,26,41,45,51,56,66,69,84,87,101,125,126,129,140,144,150,162,168,185,200,210,213,224,225,228,243,249,266,281,291,294,305,306,309,324,327,350,365,369,375,380,390,393,408,411,425,449,450,453,464,468,474,489,492,506,530,531,534,545,549,555,567,573,590,605,615,618,629,630,633,648,651,674,689,693,699,704,714,717\\
2&0,22,26,29,30,43,64,69,77,90,98,103,110,129,133,136,140,150,162,166,179,200,205,210,217,239,240,254,255,268,275,289,294,301,305,306,327,331,335,361,365,375,382,396,404,408,425,430,437,438,442,468,472,485,487,492,500,526,530,531,547,560,561,573,586,590,594,602,607,633,637,641,662,663,667,683,693,697,704,709,714\\
3&0,3,6,45,48,51,63,66,69,90,93,96,108,111,114,153,156,159,180,183,186,198,201,204,216,219,222,252,255,258,270,273,276,315,318,321,342,345,348,360,363,366,378,381,384,405,408,411,450,453,456,468,471,474,504,507,510,522,525,528,540,543,546,567,570,573,612,615,618,630,633,636,657,660,663,675,678,681,720,723,726\\
4&0,3,10,16,23,26,58,65,78,83,86,90,96,103,106,138,145,161,163,166,173,179,183,186,221,225,241,271,287,291,297,303,310,313,317,320,351,367,374,380,386,390,393,397,400,434,447,454,460,466,473,476,477,480,493,500,504,516,519,523,526,533,539,573,580,587,599,602,603,606,613,619,656,660,667,679,682,686,689,693,699\\
5&0,3,14,27,30,41,54,57,68,86,99,102,113,126,129,140,153,156,171,174,185,198,201,212,225,228,239,252,255,266,279,282,293,306,309,320,324,327,338,351,354,365,378,381,392,410,423,426,437,450,453,464,477,480,491,504,507,518,531,534,545,558,561,576,579,590,603,606,617,630,633,644,648,651,662,675,678,689,702,705,716\\
6&0,19,20,27,46,47,54,73,74,83,99,100,110,126,127,137,153,154,163,180,182,190,207,209,217,234,236,244,245,252,271,272,279,298,299,306,324,325,335,351,352,362,378,379,389,405,407,415,432,434,442,459,461,469,496,497,504,523,524,531,550,551,558,576,577,587,603,604,614,630,631,641,657,659,667,684,686,694,711,713,721\\
7&0,16,23,34,41,45,57,64,80,82,98,102,115,122,126,140,144,160,162,178,185,197,201,208,221,225,241,248,252,268,270,286,293,305,309,316,327,334,350,351,367,374,385,392,396,410,414,430,433,449,453,466,473,477,493,500,504,518,522,538,541,557,561,575,579,586,599,603,619,621,637,644,655,662,666,678,685,701,702,718,725\\
8&0,16,23,30,37,53,56,69,76,83,96,103,113,117,133,136,152,156,163,179,183,193,200,213,216,232,239,248,252,268,278,282,289,301,308,321,328,335,348,358,365,369,381,388,404,408,415,431,438,445,452,464,468,484,493,500,504,514,530,534,546,553,560,573,580,587,594,610,617,629,633,640,656,660,667,677,690,697,709,716,720\\
9&0,4,14,38,39,49,57,73,77,82,99,107,110,115,117,145,150,155,175,179,186,197,210,214,221,222,232,243,263,268,271,276,281,306,314,316,336,340,350,358,374,375,382,386,393,407,408,418,442,446,453,464,477,481,500,501,511,519,535,539,543,547,557,568,572,579,603,607,617,625,641,642,650,667,672,675,683,685,713,718,720\\
10&0,14,25,33,38,49,57,71,73,82,93,107,115,117,131,139,150,155,164,175,186,197,199,210,221,232,234,245,256,267,278,280,291,302,313,315,324,338,349,357,362,373,381,395,397,406,417,431,439,441,455,463,474,479,487,498,512,520,522,536,544,555,560,569,580,591,602,604,615,626,637,639,648,662,673,681,686,697,705,719,721\\
11&0,4,8,38,39,43,65,66,70,91,95,96,118,122,123,145,149,150,171,175,179,198,202,206,236,237,241,252,256,260,271,275,276,306,310,314,335,336,340,362,363,367,389,390,394,423,427,431,442,446,447,469,473,474,496,500,501,523,527,528,542,543,547,576,580,584,603,607,611,630,634,638,659,660,664,694,698,699,713,714,718\\
12&0,4,8,31,38,51,54,58,62,83,84,88,118,122,123,137,150,157,168,175,182,190,194,195,236,237,241,244,248,249,271,275,276,302,306,322,324,340,347,351,355,359,389,390,394,423,427,431,439,446,450,461,462,466,489,496,512,515,516,520,542,543,547,576,580,584,595,611,615,622,626,627,648,652,656,694,698,699,710,714,721\\
13&0,4,8,37,41,42,74,75,79,99,103,107,109,113,114,146,147,151,171,175,179,208,212,213,218,219,223,252,256,260,289,293,294,299,300,304,324,328,332,361,365,366,398,399,403,423,427,431,433,437,438,470,471,475,504,508,512,514,518,519,551,552,556,576,580,584,613,617,618,623,624,628,648,652,656,685,689,690,722,723,727\\
14&0,4,8,47,48,52,64,68,69,101,102,106,118,122,123,135,139,143,172,176,177,189,193,197,236,237,241,244,248,249,288,292,296,308,309,313,342,346,350,362,363,367,379,383,384,416,417,421,433,437,438,477,481,485,488,489,493,532,536,537,549,553,557,586,590,591,603,607,611,623,624,628,657,661,665,677,678,682,721,725,726\\
\bottomrule
\end{longtable}
\endgroup
The layer blocks are the consecutive integer-code intervals determined
by~\eqref{eq:f3-weights}; equivalently, $Y_j$ begins at
$\sum_{i<j}a_i$ and ends at $\sum_{i\le j}a_i-1$.

\end{document}